\documentclass{elsarticle}
\usepackage{amsmath,amssymb}
\newcommand{\PGL}{\mathop{\mathrm{PGL}}}

\newcommand{\PSL}{\mathop{\mathrm{PSL}}}

\newcommand{\Sym}{\mathop{\mathrm{Sym}}}
\newcommand{\Alt}{\mathop{\mathrm{Alt}}}
\newcommand{\crr}{\mathop{\mathrm{crr}}}
\newcommand{\fix}{\mathop{\mathrm{fix}}}
\newcommand{\Cay}{\mathop{\mathrm{Cay}}}
\newcommand{\Irr}{\mathop{\mathrm{Irr}}}
\newcommand{\Aut}{\mathop{\mathrm{Aut}}}
\newcommand{\Imm}{\mathop{\mathrm{Im}}}
\newcommand{\Ker}{\mathop{\mathrm{Ker}}}

\newtheorem{theorem}{Theorem}
\newtheorem{proposition}[theorem]{Proposition}
\newtheorem{lemma}[theorem]{Lemma}
\newtheorem{conjecture}{Conjecture}
\newproof{proof}{Proof}

\begin{document}

\title{An Erd\H{o}s-Ko-Rado theorem for the derangement graph of
  $\PGL(2,q)$ acting on the projective line} 
\author[karen]{Karen Meagher\corref{cor1}\fnref{fn1}}
\ead{karen.meagher@uregina.ca}

\address[karen]{Karen Meagher, Department of Mathematics and Statistics,\\
University of Regina, 3737 Wascana Parkway, S4S 0A4 Regina SK, Canada}

\author[pablo]{Pablo Spiga}
\ead{spiga@math.unipd.it}

\address[pablo]{Pablo Spiga, Dipartimento di Matematica Pura ed Applicata,\newline
Universit\`a degli Studi di Padova, via Trieste 63,
35121 Padova, Italy.}

\cortext[cor1]{Corresponding author}
\fntext[fn1]{Research supported by NSERC.}

\begin{abstract}
Let $G=\PGL(2,q)$ be the projective general linear group acting on the
projective line $\mathbb{P}_q$. A subset $S$ of $G$ is
intersecting if for any pair of permutations $\pi,\sigma$ in $S$,
there is a projective point $p\in \mathbb{P}_q$ such that
$p^\pi=p^\sigma$. We prove that if $S$ is intersecting, then  $|S|\leq
q(q-1)$. Also, we prove that the only sets $S$  
that meet this bound are the cosets of the stabilizer of a point of
$\mathbb{P}_q$. 
\end{abstract}


\maketitle

\section{Introduction}
The Erd\H{o}s-Ko-Rado theorem~\cite{ErKoRa} is a very important result
in extremal combinatorics. There are many different proofs and
extensions of this theorem and we refer the reader to~\cite{DeFr} for
a full account. There are also many applications of the
Erd\H{o}s-Ko-Rado theorem, for example to qualitatively independent
sets~\cite{katona:73, kleitman:73}, problems in
geometry~\cite{MR934346} and in statistics~\cite{Li}.
 
In this paper, we are concerned with an extension of the
Erd\H{o}s-Ko-Rado theorem to permutation groups. Let $G$ be a
permutation group on $\Omega$. We let $\fix(g)$ denote the number of
fixed points of the permutation $g$ of $G$. A subset $S$ of $G$ is
said to be \emph{intersecting} if $\fix(g^{-1}h)\neq 0$, for every
$g,h\in S$. As with the Erd\H{o}s-Ko-Rado theorem, we are interested
in finding the size of the largest intersecting set of $G$ and
classifying the sets that attain this bound.  This problem can be
formulated in a graph-theoretic terminology.  We denote by $\Gamma_G$
the \emph{derangement graph} of $G$, the vertices of this graph are
the elements of $G$ and the edges are the pairs $\{g,h\}$ such that
$\fix(g^{-1}h)= 0$. An intersecting set of $G$ is simply an
\emph{independent set} of $\Gamma_G$.  In~\cite{CaKu} and~\cite{LaMa}
the natural  
extension of the Erd\H{o}s-Ko-Rado theorem for the symmetric group
$\Sym(n)$ was independently proven.  Indeed, it was shown that every
independent set of permutations in $\Gamma_{\Sym(n)}$ has size at most
$(n-1)!$. Also, the only sets that meet this bound are the cosets of
the stabilizer of a point. The same result was also proved
in~\cite{GoMe} using the character theory of $\Sym(n)$.

In this paper, inspired by the approach used in~\cite{GoMe}, we prove
a similar result for the permutation group $G=\PGL(2,q)$ acting on the
projective line $\mathbb{P}_q$. We show that an independent set $S$ in
$\Gamma_{G}$ has size at most $q(q-1)$. Also, we prove that the sets
$S$ that meet this bound are the cosets of the stabilizer of a point
of $\mathbb{P}_q$. In particular, these results are a natural $q$-analogue 
of the result for the symmetric group
$\Sym(n)$. Here we report the main theorem.

\begin{theorem}\label{main}Every independent set $S$ of the derangement graph of
  $\PGL(2,q)$ acting on the projective line $\mathbb{P}_q$ has size at
  most $q(q-1)$. Equality is met if and only if $S$ is the coset
  of the stabilizer of a point.
\end{theorem}

Theorem~\ref{main} indicates that the argument described
in~\cite{GoMe} will likely be 
particularly fruitful for studying the derangement graph of a
permutation group $G$ whose character theory is well-understood. 

In Section~\ref{generalresults} we recall some general results on the
eigenvalues and the independent sets of the graph $\Gamma_G$ that will
be used throughout the paper. In Section~\ref{pgl2q} the character
table of $\PGL(2,q)$ is described. Section~\ref{aux} includes
technical lemmas that are used in Section~\ref{mainSec} to prove
Theorem~\ref{main}. Finally, in Section~\ref{comments}, we conclude
with some general remarks on the derangement graphs of $2$-transitive
and $3$-transitive groups.

\section{General results}\label{generalresults}

Let $G$ be a permutation group on $\Omega$ and $\Gamma_G$ its
derangement graph. Since the right regular representation of $G$ is a
subgroup of $\Aut(\Gamma_G)$, we see that $\Gamma_G$ is a Cayley
graph. Namely, if $D$ is the set of \emph{derangements} of $G$ (i.e. the
fixed-point-free permutations of $G$), then
$\Gamma_G$ is the Cayley graph on $G$ with connection set $D$,
i.e. $\Gamma_G=\Cay(G,D)$.  Clearly, $D$ is a union of $G$-conjugacy
classes, so $\Gamma_G$ is a normal Cayley graph. 

As usual, we simply say that the complex number $\xi$ is an
\emph{eigenvalue} of the graph $\Gamma$ if $\xi$ is an
eigenvalue of the adjacency matrix of $\Gamma$. 

In this paper, we use $\Irr(G)$ to denote the \emph{irreducible complex
characters} of a group $G$ and given $\chi\in\Irr(G)$ and  a subset $S$ of $G$
we write $\chi(S)$ for $\sum_{s\in S}\chi(s)$. In the following lemma we
recall that the eigenvalues of $\Gamma_G$ 
are determined by the irreducible complex characters of the group
$G$.

\begin{lemma}\label{eigenvalues}Let $G$ be a permutation group on
  $\Omega$ and $D$ the set of derangements of $G$. The spectrum of the
  graph $\Gamma_G$ 
  is $\{\chi(D)/\chi(1)\mid \chi\in \Irr(G)\}$. Also, if
  $\tau$ is an eigenvalue of $\Gamma_G$ and $\chi_1,\ldots,\chi_s$
  are the irreducible characters of $G$ such that
  $\tau=\chi_i(D)/\chi_i(1)$, then the dimension of the $\tau$ 
  eigenspace of $\Gamma_G$ is $\sum_{i=1}^s\chi_i(1)^2.$
\end{lemma}
\begin{proof}
Since $\Gamma_G$ is a normal Cayley graph, the result follows
from~\cite{Ba}. 
\end{proof}

Since $\Gamma_G$ is a Cayley graph, the $\mathbb{C}$-vector space with
basis elements labelled by the vertices of $\Gamma_G$ is the
$\mathbb{C}$-vector space underlying the group algebra
$\mathbb{C}G$. If $S$ is a subset of $G$, we simply write $v_S$ for
the vector $\sum_{g\in S}g$ of $\mathbb{C}G$ (so $v_S$ is the
\emph{characteristic vector} for the set $S$). In particular, $v_G$ is
the all-1 vector of $\mathbb{C}G$. Finally, we recall Hoffman's
bound for the size of an independent set of $\Gamma_G$ and a
consequence for when equality holds in this bound.

\begin{lemma}\label{bound}
Let $G$ be a permutation group on $\Omega$. Let $S$ be an
independent set of $\Gamma_G$ and $\tau$ be the minimum eigenvalue of
$\Gamma_G$. Assume that the valency of $\Gamma_G$ is
$d$. Then $|S|\leq |G|/(1-\frac{d}{\tau})$. 
 If the equality is met, then $v_S-\frac{|S|}{|G|}v_G$ is an eigenvector
 of $\Gamma_G$ with eigenvalue $\tau$.
\end{lemma}

\begin{proof}
Set $v=|G|$, and $M=A-\tau I-(d-\tau)/vJ$, where $A$ is the
adjacency matrix of $\Gamma_G$ and $J$ is the all-1 matrix. By
definition of $\tau$ and by
construction, $M$ is a positive semidefinite symmetric matrix. Hence,
given a set $S \subseteq G$ of size $s$, it is true that 
$$0\leq v_S^T  Mv_S=v_S^T Av_S-\tau v_S^T
v_S-\frac{d-\tau}{v} v_S^T Jv_S 
=v_S^T Av_S-\tau s-\frac{d-\tau}{v} s^2.$$ 
If $S$ is an independent set of $\Gamma_G$, then
$v_S^T Av_S=0$ and the previous inequality yields the first part of
the lemma.

Now, suppose that equality holds. Then $v_S^T Mv_S=0$. Since $M$ is
positive semidefinite, we obtain $Mv_S=0$. Therefore 

$$A\left(v_S-\frac{s}{v}v_G\right)=\tau\left(v_S-\frac{s}{v}v_G\right),$$
and the proof is completed.
\end{proof}

\section{$\PGL(2,q)$} \label{pgl2q}

We recall that the character table of $\PGL(2,q)$ was computed by
Jordan and Schur in $1907$
and can be found in many textbooks, see~\cite{JamesLiebeck}. Also,
the character tables of $\PGL(3,q)$ and $\PGL(4,q)$ were found by R. 
Steinberg and the character table of $\PGL(n,q)$ was finally determined
by J. Green in $1955$ in the celebrated paper~\cite{Gr}.

We give the character table of $\PGL(2,q)$, first for $q$ even and
second for $q$ odd. We note that by abuse of terminology we often
refer to the elements of $\PGL(2,q)$ as matrices.

\begin{table}[!h]
\begin{tabular}{|c|c|cccc|}\hline
            &Name        &$1$  &$u$    &$d_x$    &$v_r$\\\hline 
            &Nr.         &$1$  &$1$    &$\frac{q}{2}-1$&$\frac{q}{2}$\\\hline
            &Size        &$1$  &$q^2-1$&$q(q+1)$ &$q(q-1)$\\\hline
Name        &Nr.         &     &       &         &          \\\hline
$\lambda_1$ &$1$         &$1$  &   $1$ &     $1$ &    $1$   \\\hline
$\psi_1$    &$1$         &$q$  & $0$   &     $1$ &   $-1$   \\\hline
$\eta_\beta$&$\frac{q}{2}$  &$q-1$&$-1$&$0$&$-\beta(r)-\beta(r^{-1})$\\\hline
$\nu_\gamma$&$\frac{q}{2}-1$&$q+1$&$1$ &$\gamma(x)+\gamma(x^{-1})$&$0$\\\hline 
\end{tabular}
\caption{Character table of $\PGL(2,q)$, for $q$ even.}\label{Tb2}
\end{table}

\begin{table}[!h]
\begin{tabular}{|c|c|cccccc|}\hline
             & Name&$1$  &$u$    &$d_x$     &$d_{-1}$   &$v_r$&$v_{i}$\\\hline 
             & Nr.&$1$&$1$&$\frac{q-3}{2}$ &$1$&$\frac{q-1}{2}$&$1$\\\hline
&Size&$1$&$q^2\!-\!1$&$q(q+1)$&$\frac{q(q+1)}{2}$&$q(q-1)$&$\frac{q(q-1)}{2}$\\\hline
Name         & Nr.  &     &       &          &          &&\\\hline
$\lambda_1$  &  $1$    &$1$  &   $1$
&$1$&$1$&$1$&$1$\\\hline
$\lambda_{-1}$  &  $1$    &$1$  &   $1$
&$\varepsilon(x)$&$\varepsilon(-1)$&$\varepsilon(r)$&$\varepsilon(i)$\\\hline
$\psi_1$&$1$&$q$&$0$&$1$&$1$&$-1$&$-1$\\\hline
$\psi_{-1}$&$1$&$q$&$0$&$\varepsilon(x)$&$\varepsilon(-1)$&$-\varepsilon(r)$&$-\varepsilon(i)$\\\hline
$\eta_\beta$&$\frac{q-1}{2}$&$q-1$&$-1$&$0$&$0$&$-\beta(r)\!\!-\!\beta(r^{-1})$&$-2\beta(i)$\\\hline
$\nu_\gamma$&$\frac{q-3}{2}$&$q+1$&$1$&$\gamma(x)\!+\!\gamma(x^{-1})$&$2\gamma(-1)$&$0$&$0$\\\hline  
\end{tabular}
\caption{Character table of $\PGL(2,q)$, for $q$ odd.}\label{Tb1}
\end{table}

We briefly explain the notation used, but refer the reader
to~\cite{JamesLiebeck} for full details. We start by describing the 
conjugacy classes. The elements of $\PGL(2,q)$
can be collected in four sets: the set consisting only of the
identity element, the set consisting of the non-scalar
matrices with only one eigenvalue,  the set
consisting of the matrices with two distinct 
eigenvalues in $\mathbb{F}_q$ and the set of matrices with no
eigenvalues in $\mathbb{F}_q$. The non-scalar matrices with only one
eigenvalue form a conjugacy class of size $q^2-1$ and are represented by
a unipotent matrix $u$. Every matrix with two distinct eigenvalues in
$\mathbb{F}_q$ is conjugate to a diagonal matrix $d_x$, with $x$ and $1$
along the main diagonal. Now, $d_x$ and $d_y$ are conjugate if and
only if $y=x$ or $y=x^{-1}$.  So, the label $x$ in the character table
of $\PGL(2,q)$ for $d_x$ represents
an element  of $\mathbb{F}_q\setminus\{0,1\}$ up to
inversion.  Now, for $q$ odd, let 
$i\in\mathbb{F}_{q^2}$ be an element of order $2$ in
$\mathbb{F}_{q^2}^*/\mathbb{F}_q^*$. The matrices with no eigenvalues
in  $\mathbb{F}_q$ are conjugate in $\PGL(2,q^2)$ to a diagonal
matrix, with $r$ and 
$r^q$ along the main diagonal (for $r\in
\mathbb{F}_{q^2}\setminus\mathbb{F}_q$). Each of these matrices is
conjugate in $\PGL(2,q)$ to the matrix $v_r$ depicted below.  Also,
$v_x$ and $v_y$ are 
conjugate if and only if 
$y\mathbb{F}_q^*=x\mathbb{F}_q^*$ or
$y\mathbb{F}_q^*=x^{-1}\mathbb{F}_q^*$. So, the label $r$ in the
character table of 
$\PGL(2,q)$ for $v_r$ represents an element of
$\mathbb{F}_{q^2}^*/\mathbb{F}_q^*$ up to inversion.  So, \[
u=\left[
\begin{array}{cc}
1&1\\
0&1
\end{array}
\right],
d_x=\left[
\begin{array}{cc}
x&0\\
0&1
\end{array}
\right],
v_r=\left[
\begin{array}{cc}
0&1\\
-r^{1+q}&r+r^q
\end{array}
\right].
\]


We note here that in the permutation group $\PGL(2,q)$ acting on the
projective line $\mathbb{P}_q$, the permutations in the conjugacy
class represented by $u$ fix exactly one point, the permutations in
the conjugacy classes represented by $d_x$ fix exactly $2$ points
while the permutations in the conjugacy classes represented by
$v_r$ fix no point. 

Next we describe the characters of $\PGL(2,q)$ by describing the maps
$\varepsilon$, $\gamma$ and $\eta$. The map $\varepsilon$ is defined by
$\varepsilon(x)=1$ if $d_x\in \PSL(2,q)$ and $\varepsilon(x)=-1$
otherwise. Similarly, $\varepsilon(r)=1$ if $v_r\in \PSL(2,q)$ and
$\varepsilon(r)=-1$ otherwise.  So, $\lambda_{-1}:\PGL(2,q) \to\{\pm 1\}$ is
the non-principal linear character.  The letter $\gamma$ represents a
group homomorphism $\gamma:\mathbb{F}_q^*\to \mathbb{C}$ of order
greater than $2$. Also, $\nu_{\gamma_1}=\nu_{\gamma_2}$ if and only if
$\gamma_2=\gamma_1$ or $\gamma_2=\gamma_1^{-1}$. So, the label
$\gamma$ runs through the homomorphisms $\gamma:\mathbb{F}_q^{*}\to
\mathbb{C}$ of order greater than $2$ up to inversion. The letter
$\beta$ stands for a group homomorphism
$\beta:\mathbb{F}_{q^2}^{*}/\mathbb{F}_q^*\to \mathbb{C}$ of order
greater than $2$. Also, $\eta_{\beta_1}=\eta_{\beta_2}$ if and only if
$\beta_2=\beta_1$ or $\beta_2=\beta_1^{-1}$. So, the label $\beta$
runs through the homomorphisms
$\beta:\mathbb{F}_{q^2}^{*}/\mathbb{F}_q^*\to \mathbb{C}$ of order
greater than $2$ up to inversion. Note that given a non-principal
linear character $\psi:C\to\mathbb{C}$ of a cyclic group $C$, we get
$\sum_{c\in C}\psi(c)=0$. In particular, given $\gamma$ and $\beta$ as
in Table~\ref{Tb2} and Table~\ref{Tb1}, we obtain
$\sum_{x\in\mathbb{F}_q^*}\gamma(x)=0$ and
$\sum_{r\mathbb{F}_q^*\in\mathbb{F}_{q^2}^*/\mathbb{F}_q^*}\beta(r)=0$.

In the following we simply denote by $G_{q}$ the permutation group
$\PGL(2,q)$ acting on the projective line $\mathbb{P}_q$, and $D$ the
set of derangements of $G_q$.

Using Table~\ref{Tb2} (for $q$ even), Table~\ref{Tb1}
(for $q$ odd) and the fact that the elements in the conjugacy classes
labelled by $v_r$ are the derangements
of $G_q$, it is possible to compute $\chi(D)/\chi(1)$, for all $\chi\in
\Irr(G_q)$. For example, if $\chi = \lambda_1$ then we easily see that 
\[
\frac{\lambda_1(D)}{\lambda_1(1)}=\frac{q(q-1)(q-1)}{2}+\frac{q(q-1)}{2}
=\frac{q^2(q-1)}{2},
\] 
which is equal to $|D|$, is an eigenvalue of $\Gamma_{G_q}$.
The only cases where this
is not trivial is when $\chi=\eta_\beta$ and $\chi=\lambda_{-1}$. If
$\chi=\eta_\beta$, then we have
\[
\frac{\chi(D)}{\chi(1)}=
\frac{q(q-1)}{q-1}\bigg(
\sum_{\stackrel{r\mathbb{F}_q^*\in\mathbb{F}_{q^2}^*/\mathbb{F}_q^*}{r\notin\mathbb{F}_q^*}}
-\beta(r)\bigg)=
 -q \bigg(
\sum_{\stackrel{r\mathbb{F}_q^*\in\mathbb{F}_{q^2}^*/\mathbb{F}_q^*}{r\notin\mathbb{F}_q^*}}
\beta(r)\bigg)=
q
\]
since $\mathbb{F}_{q^2}^*/\mathbb{F}_q^*$ is a cyclic group. By direct
calculation or using the description of the conjugacy classes of
$\PSL(2,q)$ given in~\cite{JamesLiebeck}, we see that
$\PSL(2,q)$ contains $q(q-1)^2/4$ 
derangements. As $|D|=q^2(q-1)/2$, we
get that $\PGL(2,q)\setminus \PSL(2,q)$ contains $q(q^2-1)/4$
derangements. So, if 
$\chi=\lambda_{-1}$, we have 
$$\frac{\lambda_{-1}(D)}{\lambda_{-1}(1)}=\frac{q(q-1)^2}{4}-
\frac{q(q^2-1)}{4}=-\frac{q(q-1)}{2}.$$  

In particular, we obtain the following table:
\begin{table}[!h]
\begin{tabular}{|l|cccccc|}\hline
Character
&$\lambda_1$&$\lambda_{-1}$&$\psi_1$&$\psi_{-1}$&$\eta_\beta$
&$\nu_\gamma$\\\hline 
Eigenvalue &$\frac{q^2(q-1)}{2}$&$\frac{-q(q-1)}{2}$&$\frac{-q(q-1)}{2}$& 
$\frac{q-1}{2}$&$q$&$0$ \\ \hline
Dimension & $1$ & $1$ & $q^2$ & $q^2$ & $\frac{(q-1)^3}{2}$ & $\frac{(q+1)^2(q-3)}{2}$ \\ \hline
\end{tabular}
\caption{Eigenvalues of $\Gamma_{G_q}$\label{evaluesPGL2q}}
\end{table}

In summary, the valency of the graph $\Gamma_{G_q}$ is $q^2(q-1)/2$ and the
minimum eigenvalue $\tau$ is $-q(q-1)/2$.
If $q$ is even, then $\psi_1$ is the only character $\chi$
such that $\tau=\chi(D)/\chi(1)$. If $q$ is odd, then $\lambda_{-1}$
and $\psi_1$ are the only characters $\chi$ such that
$\tau=\chi(D)/\chi(1)$.

\begin{lemma}\label{boundPGL2q}
An independent set of maximal size in $\Gamma_{G_q}$ has size $q(q-1)$. 
\end{lemma}
\begin{proof}
The coset of the stabilizer of a point of $G_{q}$ is an independent
set in $\Gamma_{G_q}$ of size $|G_{q}|/(q+1)=q(q-1)$. From the
eigenvalues of $\Gamma_{G_q}$ and Lemma~\ref{bound} we see that
such an independent set is an independent set of maximal size. 
\end{proof}

Similar to the case of the symmetric group (and also the standard
Erd\H{o}s-Ko-Rado theorem for sets), finding the bound in Theorem~\ref{main} is
not difficult, but it is in the characterization of the sets that meet
this bound that the difficulty lies. Indeed, it is not difficult to
also establish a similar bound on the size of the independent sets in
the derangement graph for the group $\PSL(2,q)$.

\begin{lemma}\label{boundPSL2q}
An independent set of maximal size in $\Gamma_{\PSL(2,q)}$ has size $q(q-1)/2$. 
\end{lemma}
\begin{proof}
  This can be proved using Lemma~\ref{bound} with the information on
  the character table of $\PSL(2,q)$ in~\cite{JamesLiebeck} and the
  fact that a point-stabilizer in $\PSL(2,q)$ has size $q(q-1)/2$.
\end{proof}

The next lemma will be used in Lemma~\ref{rankA} to limit the search
of independent sets of maximal size in $\Gamma_{G_q}$.

\begin{lemma}\label{fix}Assume $q$ odd. If $S$ is an independent set of maximal
  size of $\Gamma_{G_q}$, then $\lambda_{-1}(S)=0$. 
\end{lemma}
\begin{proof}
  By Lemma~\ref{boundPGL2q}, we have $|S|=q(q-1)$.  Consider the two sets
  $S_+=S\cap \PSL(2,q)$ and $S_{-}=S\setminus S_+$.  Clearly $S_+$ is
  an independent set in $\Gamma_{\PSL(2,q)}$ and, for $g\in
  G_q\setminus\PSL(2,q)$, the set $gS_{-}$ is also an independent set
  in $\Gamma_{\PSL(2,q)}$. Thus we obtain that $|S_+|=|S_{-}|=q(q-1)/2$ and
  in particular $\lambda_{-1}(S)=|S_+|-|S_{-}|=0$.
\end{proof}

\section{Auxiliary lemmas}\label{aux}

Consider the $\{0,1\}$-matrix $A$, where the rows are indexed by the
elements of $G_{q}$, the columns are indexed by the ordered pairs of
points of $\mathbb{P}_q$ and $A_{g,(p_1,p_2)}=1$ if and only
if $p_1^g=p_2$. In particular, $A$ has $|G_q|=q(q^2-1)$ rows and
$|\mathbb{P}_q|^2=(q+1)^2$ columns.

The entry $( A^T A)_{(p_1,q_1),(p_2,q_2)}$ equals the number of permutations of
$G_{q}$ mapping $p_1$ into $q_1$ and $p_2$ into $q_2$. Since
$G_{q}$ is $2$-transitive, we get by a simple counting argument that 
\[
(A^T A)_{(p_1,q_1),(p_2,q_2)}=
\left\{
\begin{array}{lcl}
q(q-1)&&\textrm{if }p_1=p_2 \textrm{ and }q_1=q_2,\\
q-1&&\textrm{if }p_1\neq p_2
\textrm{ and }q_1\neq q_2,\\ 
0&&\textrm{otherwise.}
\end{array}
\right.
\]
This shows that with the proper ordering of the columns of $A$,
$$A^T A=q(q-1)I_{(q+1)^2}+(q-1)(J_{q+1}-I_{q+1})\otimes
(J_{q+1}-I_{q+1}),$$ 
(in here $I_{q+1},J_{q+1}$ denote the identity
matrix and the all-$1$ matrix of size $q+1$, respectively).  The
matrix $J_{q+1}$ has eigenvalue $0$ (with multiplicity $q$) and $q+1$
(with multiplicity $1$). Hence $(J_{q+1}-I_{q+1})\otimes
(J_{q+1}-I_{q+1})$ has eigenvalue $1$ (with multiplicity $q^2$), $-q$
(with multiplicity $2q$), and $q^2$ (with multiplicity $1$). So, $A^T
A$ is diagonalizable with eigenvalues $q(q-1)+(q-1)q^2=|G_{q}|$ (with
multiplicity $1$), $q(q-1)+(q-1)=q^2-1$ (with multiplicity $q^2$) and
$q(q-1)-(q-1)q=0$ 
(with multiplicity $2q$). This shows that the kernel of $A^T A$ has 
dimension $2q$. We now determine the kernel of $A$. 

Let $V$ be the $\mathbb{C}$-vector space whose basis consists of all
$e_{(x,y)}$, where $(x,y)$ is an ordered pair of elements of
$\mathbb{P}_q$. Consider the following two subspaces of $V$
\begin{align*}
V_1&=\big\langle \sum_{x\in \mathbb{P}_q}(e_{(p_1,x)}-e_{(p_2,x)})
\mid p_1,p_2\in \mathbb{P}_q\big\rangle \\
V_2&=\big\langle \sum_{x\in \mathbb{P}_q}(e_{(x,p_1)}-e_{(x,p_2)})\mid
p_1,p_2\in \mathbb{P}_q\big\rangle. 
\end{align*} 
Note that by construction, $V_1$ and $V_2$ have dimension $q$ and are
$G_q$-modules. As 
$V_1$ and $V_2$ are orthogonal, 
we have $V_1\cap V_2=0$. Using the
 definition of $A$, it is easy to check  that $V_1\oplus V_2$ is
 contained in the 
kernel of $A$. Since the kernel of $A^T A$ has dimension $2q$, we obtain that
$V_1\oplus V_2$ is the kernel of $A$. In particular, we proved the
first part of the following lemma.

\begin{lemma}\label{rankA}The matrix $A$ has rank $q^2+1$ and the
  kernel of $A$   is $V_1\oplus V_2$. Also, the vector space spanned
  by the columns of $A$ equals the
  vector space spanned by the characteristic vectors of the independent
  sets of size $q(q-1)$ of $\Gamma_{G_q}$.
\end{lemma}

\begin{proof}
The claims on the rank and on the kernel follow from the
previous discussion.

Note that the columns of $A$ are the characteristic vectors of cosets
of stabilizers of points of $G_q$. In particular, these columns are
characteristic vectors of independent sets of size $q(q-1)$ in
$\Gamma_{G_q}$. By taking the sum of all the columns in $A$, we see
that the all-$1$ vector is also in the column space of $A$.  

Let $S$ be any independent set of $\Gamma_{G_q}$ of size $q(q-1)$.  By
Lemma~\ref{bound} and~\ref{boundPGL2q} the characteristic vector of
$S$ is in the direct sum of the $q^2(q-1)/2$-eigenspace and the
$-q(q-1)/2$-eigenspace.

If $q$ is even, from the multiplicities of the eigenvalues in
Table~\ref{evaluesPGL2q} it is clear that the vector space spanned by
the characteristic vectors of independent sets of size $q(q-1)$ has
dimension at most $q^2+1$. So the lemma follows.

If $q$ is odd then Lemma~\ref{fix} yields that $v_S$ is orthogonal to
the eigenspace arising from the character $\lambda_{-1}$. So $v_S$
must lie in the direct sum of the eigenspaces arising from
$\lambda_1$ and $\psi_1$. Again, the multiplicities of the eigenvalues
in Table~\ref{evaluesPGL2q} show that the vector space spanned by the
characteristic vectors of independent sets of size $q(q-1)$ has
dimension at most $q^2+1$ and the lemma follows.
\end{proof}

Now we fix a particular ordering of the rows of $A$ so that the first
row is labelled by the identity element of $G_{q}$, then we label the
next $q^2(q-1)/2$ rows by the derangements of $G_{q}$ and the final
$(q^2-2)(q+1)/2$ 
rows are labelled by the remaining permutations. Similarly, we fix a particular
ordering of the columns of $A$ so that the first $q+1$ columns are
labelled by the 
ordered pairs of the form $(p,p)$ and then the last $q(q+1)$ columns are
labelled by the ordered pairs of the form $(p_1,p_2)$, with $p_1\neq
p_2$. With this ordering, we get that the matrix $A$ is a block matrix. Namely,
\[
A=\left(
\begin{array}{cc}
1&0\\
0&M\\
B&C
\end{array}
\right).
\] 
In particular, the rows of the submatrix $M$ are labelled by the derangements of
$G_{q}$ and the columns of $M$ are labelled by the ordered pairs of
distinct elements of $\mathbb{P}_q$.

Now, set $N=M^T M$. In particular, $N$ is a square $(q+1)q$
matrix whose rows and columns are indexed by the ordered pairs
of distinct points. 

From now on, we identify the elements of $\mathbb{P}_q$ with the 
elements  of the set
$\mathbb{F}_q\cup\{\infty\}$. Namely, the point $[1,a]$ corresponds to
the element $a$ of $\mathbb{F}_q$ and the point $[0,1]$ corresponds to
the element $\infty$. Now, given four distinct points
$\alpha,\beta,\gamma,\delta$, we recall that the \emph{cross-ratio},
denoted by $\crr(\alpha,\delta,\gamma,\beta)$, is defined
by 
$$\frac{\alpha-\gamma}{\alpha-\beta}\frac{\delta-\beta}{\delta-\gamma}. $$
We recall that the cross-ratio is $G_q$-invariant, i.e. if
$g\in G_q$ and $\alpha,\beta,\gamma,\delta\in \mathbb{P}_q$, then
$\crr(\alpha^g,\delta^g,\gamma^g,\beta^g)=\crr(\alpha,\delta,\gamma,\beta)$.

In the following proposition we prove that the entries of the matrix
$N$ are determined by the cross-ratio.
 
\begin{proposition}\label{crossratio}
If $q$ is even, then
\[
N_{(\alpha,\beta),(\gamma,\delta)}=
\left\{
\begin{array}{lcl}
q(q-1)/2&&\textrm{if }\alpha=\gamma \textrm{ and }\beta=\delta,\\
0&&\textrm{if }\alpha=\gamma \textrm{ and }\beta\neq \delta,\\
0&&\textrm{if }\alpha\neq\gamma \textrm{ and }\beta=\delta,\\
0&&\textrm{if }\alpha=\delta \textrm{ and }\beta=\gamma,\\
q/2&&\textrm{otherwise.}
\end{array}
\right.
\]
If $q$ is odd, then
\[
N_{(\alpha,\beta),(\gamma,\delta)}=
\left\{
\begin{array}{lcl}
q(q-1)/2&&\textrm{if }\alpha=\gamma \textrm{ and }\beta=\delta,\\
0&&\textrm{if }\alpha=\gamma \textrm{ and }\beta\neq \delta,\\
0&&\textrm{if }\alpha\neq\gamma \textrm{ and }\beta=\delta,\\
(q-1)/2&&\textrm{if }\crr(\alpha,\delta,\gamma,\beta) \text{ is a
  square in }\mathbb{F}_q,\\
(q+1)/2&&\textrm{if }\crr(\alpha,\delta,\gamma,\beta) \text{ is not a
  square in }\mathbb{F}_q.\\
\end{array}
\right.
\]
\end{proposition}
\begin{proof}
Note that  $N_{(\alpha,\beta),(\gamma,\delta)}$ is the number of
  derangements mapping $\alpha$ to $\beta$ and $\gamma$ to
  $\delta$. Write $n$ for $N_{(\alpha,\beta),(\gamma,\delta)}$. If
  $\alpha=\gamma$ and $\beta=\delta$, then 
  $n$ is the number of derangements mapping $\alpha$ to
  $\beta$. Since $G_q$ is transitive of degree $q+1$ and since
  $G_q$    contains $q^2(q-1)/2$ derangements, we have
  $n=q^2(q-1)/2q=q(q-1)/2$.   

If $\alpha=\gamma$ and $\beta\neq \delta$ or if $\alpha\neq
  \gamma$ and $\beta=\delta$, then clearly $n=0$.

From now on we can assume that $\alpha\neq \gamma$ and $\beta\neq
\delta$. Assume $\alpha=\delta$ and
  $\beta=\gamma$. Since $G_q$ is $2$-transitive, without loss of
generality, we can assume that $\alpha=0$ and $\beta=\infty$. In this
case, the elements $g$ such that $0^g=\infty$ and $\infty^g=0$ are the
matrices of the form
\[
g= \left[
\begin{array}{cc}
0&\lambda\\
1&0
\end{array}
\right] \quad \textrm{ with } \lambda\in\mathbb{F}_q^*.
\]  
Further, $g$ is a derangement if and only if $g$ has no eigenvalue in
$\mathbb{F}_q$, i.e.  the characteristic polynomial
$p_\lambda(t)=t^2-\lambda$ of $g$ is irreducible over
$\mathbb{F}_q$. If $q$ is even, then $p_\lambda(t)$ is reducible for
every value of $\lambda$, and so $n=0$. If $q$ is odd, then
$\mathbb{F}_q^*$ has $(q-1)/2$ non-square elements. Thence there exist
$(q-1)/2$ values of $\lambda$ such that $p_\lambda(t)$ is irreducible
over $\mathbb{F}_q$, and so $n=(q-1)/2$. 
Note that $\crr(\alpha,\delta,\gamma,\beta)=\crr(0,0,\infty,\infty)=1$
is a square in $\mathbb{F}_q$.

{}From now on we can assume that $|\{\alpha,\beta,\gamma,\delta\}|\geq
3$.

As $N$ is symmetric, up to interchanging the pairs 
$(\alpha,\beta)$, $(\gamma,\delta)$, we may
assume that $\beta\neq \gamma$.  Since $G_q$ is $3$-transitive, without
loss of generality, we may assume that $\alpha=0$, $\beta=1$,
$\gamma=\infty$ and $\delta=d$, for some $d\in
\mathbb{F}_q\setminus\{1\}$. The elements $g$ such that $0^g=1$ 
and $\infty^g=d$ are the matrices of the form
\[
\left[
\begin{array}{cc}
1&1 \\
\lambda&\lambda d
\end{array}
\right] \quad \textrm{with } \lambda\in\mathbb{F}_q^*. 
\]  
The permutation $g$ is a derangement if and only if $g$ has no
eigenvalue in $\mathbb{F}_q$, i.e. the characteristic polynomial
$p_\lambda(t)=t^2-(1+\lambda d)t+\lambda d-\lambda$ of $g$ is
irreducible over $\mathbb{F}_q$. Now, we determine the number of
values of $\lambda$ such that $p_\lambda(t)$ is irreducible. Assume
that $p_\lambda(t)$ is reducible over $\mathbb{F}_q$ with roots
$a,b\in\mathbb{F}_q$. So, $p_\lambda(t)=(t-a)(t-b)$. This yields
$a+b=1+\lambda d$ and $ab=\lambda d-\lambda$. We get
$b=(1+\lambda-a)/(1-a)$ (note that $a$ cannot be $1$, because
otherwise $\lambda d-\lambda=b=\lambda d$, which yields $\lambda=
0$). From this we obtain $\lambda=(a-a^2)/(d-1-ad)$ (note that
$d-1-ad\neq 0$, because otherwise $a\in\{0,1\}$, which yields
$\lambda=0$). Consider the function $\varphi:\mathbb{F}_q\backslash
X\to \mathbb{F}_q$, where $\varphi(a)=(a-a^2)/(d-1-ad)$ and
$X=\{a\mid d-1-ad= 0\}$. Note that $X=\{1-d^{-1}\}$ if $d\neq 0$, and
$X=\emptyset$ if $d=0$. 

We have proved so far that $p_\lambda(t)$ is reducible if and only if
$\lambda$ lies in the image of $\varphi$. We now compute the size of
$\Imm\varphi$. It is easy to check that $\varphi(a_1)=\varphi(a_2)$ if
and only if $a_2=a_1$ or $a_2=(1-a_1)(d-1)/(d-1-a_1d)$. This shows
that the fiber of $\varphi(a_1)$ contains two points if $a_1\neq
(1-a_1)(d-1)/(d-1-a_1d)$ and only one point if $a_1=
(1-a_1)(d-1)/(d-1-a_1d)$. Note that $a_1=(1-a_1)(d-1)/(d-1-a_1d)$ if
and only if $a_1^2d-2(d-1)a_1+(d-1)=0$.  

If $q$ is even and $d\neq 0$, then this happens if $a_1^2=d^{-1}-1$
(i.e. for a unique value of $a_1$). Thence the image of $\varphi$
contains $(q-1-1)/2+1=q/2$ elements, and $n=|\mathbb{F}_q\setminus
\Imm\varphi|=q/2$. Similarly, if $q$ is even and $d=0$, the image of
$\varphi$ contains $q/2$ elements, and
$n=|\mathbb{F}_q\setminus\Imm\varphi|=q/2$.  

If $q$ is odd, then $a_1^2d-2(d-1)a_1+(d-1)=0$ if the discriminant
$(d-1)^2-d(d-1)=(1-d)$ is a square (in this case there are two
distinct solutions for $a_1$).  So, if $1-d$ is a square, then the
image of $\varphi$ contains $(q-1-2)/2+2=(q+1)/2$ elements and so
$n=|\mathbb{F}_q\setminus\Imm\varphi|=(q-1)/2$. But, if $1-d$ is not a
square, then the image of $\varphi$ contains $(q-1)/2$ elements and so
$n=|\mathbb{F}_q\setminus\Imm\varphi|=(q+1)/2$. 
Finally, we note that $\crr(0,d,\infty,1)=1-d$.
\end{proof}

In the next proposition we  use the character table of $G_q$ to
find the rank of $M$. 
\begin{proposition}\label{rankM}The matrix $M$ has rank $q(q-1)$.
\end{proposition}

\begin{proof}
Let $\Omega$ be the set of ordered pairs of distinct elements of
$\mathbb{P}_q$ and $V$ be the vector space with  basis
$\{e_\omega\}_{\omega\in \Omega}$. Clearly, $V$ is a
$G_q$-module. Namely, $V$ is the permutation module of the action of
$G_q$ on $\Omega$. Let $\pi$ be the character afforded by $V$, so
\[
\pi(g)=|\{\omega\in\Omega\mid \omega^g=\omega\}|.
\]

We have $\pi(1_{G_q})=q(q+1)$, $\pi(g)=2$ for every element $g$
conjugate to $d_x$ (for some $x$) and $\pi(g)=0$ otherwise. As
$\pi=\sum_{\chi\in\Irr(G_q)}\langle \chi,\pi\rangle \chi$, by direct
calculation of $\langle \chi,\pi\rangle$ with the Tables~\ref{Tb2}
and~\ref{Tb1}, we get that
\begin{align*}
\pi&=\lambda_1+2\psi_1+\sum_{\beta}\eta_\beta+\sum_{\gamma}\nu_\gamma,
\quad\textrm{for }q\textrm{ even},\\
\pi&=\lambda_1+2\psi_1+\psi_{-1}+\sum_{\beta}\eta_\beta+\sum_{\gamma}\nu_\gamma,
\quad\textrm{for }q\textrm{ odd}.
\end{align*}
Let $C\subseteq \Irr(G_q)$ be the set of constituents of
$\pi$. Then $V=\oplus_{\chi\in C}V_\chi$, where $V_\chi$ is an irreducible
$G_q$-submodule of $V$, unless $\chi=\psi_1$, and $V_{\psi_1}$ is the sum of
two isomorphic irreducible $G_q$-submodules of $V$ of dimension $q$. Clearly,
$V_{\psi_1}\cong V_1\oplus V_2$ (see Lemma~\ref{rankA}).

Again we use the matrix $N = M^T M$ and in order to prove that
$M$ has rank $q(q-1)$ it suffices to prove that $N$ has rank $q(q-1)$.
By Lemma~\ref{rankA}, we get that $V_{\psi_1}$ is contained in the
kernel of $N$. Also, as
$N_{\omega_1^g,\omega_2^g}=N_{\omega_1,\omega_2}$ for every $g\in
G_q$, we obtain that every eigenspace of $N$ is a $G_q$-submodule of
$V$.  Therefore, since for $\chi\neq \psi_1$ the module $V_{\chi}$
is irreducible, we get that $V_\chi$ is an eigenspace of $N$. Thus, to
conclude the proof it suffices to show that the eigenvalue $s_\chi$ of
the eigenspace $V_\chi$ is not zero, for all $\chi\neq \psi_1$.

By Wedderburn Theorem~\cite{Is}, we get that
$\mathbb{C}G_q=\oplus_{\chi\in\Irr(G)}I_\chi$, where $I_\chi$ are
minimal two-sided ideals of the semisimple algebra
$\mathbb{C}G_q$. Also, each $I_\chi$ is generated (as an ideal) by the
idempotent
$E_\chi=\frac{\chi(1)}{|G_q|}\sum_{g\in
  G}\chi(g^{-1})g$. Set $v_\chi=\sum_{g\in
  G_q}\chi(g^{-1})e_{(0^g,\infty^g)}$. 
As $V_\chi=VI_\chi$ and
$v_\chi=\frac{|G_q|}{\chi(1)}e_{(0,\infty)}E_\chi\in V_\chi$,  we obtain
that $v_\chi$ is an eigenvector of $N$ with eigenvalue
$s_\chi$. Note that, given $\chi\in C$, we have 
\begin{eqnarray*}
(\dag)&&(Nv_\chi)_{(0,\infty)}=\!\!\sum_{(a,b)\in
  \Omega}\!N_{(0,\infty),(a,b)}(v_\chi)_{(a,b)}=\sum_{(a,b)\in
  \Omega}
\!\sum_{\stackrel{g \textrm{ s.t. }}{0^g=a,\infty^g=b}}
\!\!\!\chi(g^{-1})N_{(0,\infty),(a,b)}\\
(\ddag)&&(v_\chi)_{(0,\infty)}=
\!\!\sum_{\stackrel{g \textrm{ s.t.}}{0^g=0,\infty^g=\infty}}
\!\!\!\chi(g^{-1})=(q-1)\langle
\mathrm{Res}_T(\chi),1\rangle=(q-1)\langle\chi,\pi\rangle=q-1,
\end{eqnarray*}
where  $T$ is the stabilizer in $G_q$ of $0,\infty$ and 
$\mathrm{Res}_T(\chi)$ is the restriction of $\chi$ to $T$ (note that
in the fourth equality in $(\ddag)$ we are using Frobenius Reciprocity).

In the rest of the proof, 
we do not determine (for $q$ odd) the eigenvalue $s_\chi$ of $v_\chi$, but we
simply prove that $s_\chi>0$, for $\chi\neq \psi_1$. If $\chi=\lambda_1$, then
by $(\ddag)$ the vector $v_\chi$ is $(q-1)$ times the all-$1$ vector. 
By Proposition~\ref{crossratio},
 $N$ is a stochastic
matrix with row sum $q(q^2-1)/2$, so $s_{\lambda_1}=q(q^2-1)/2$. Now,
for the remaining characters in $C$ we distinguish two cases 
depending on whether $q$ is even or $q$ is odd.

Assume $q$ even. Let $\chi\in C$, with $\chi\neq \lambda_1,\psi_1$.
 Now, as $\sum_{g\in G_q}\chi(g^{-1})=0$, subtracting
 $q/2\sum_{g\in G_q}\chi(g^{-1})$ from $(\dag)$ and using
 Proposition~\ref{crossratio}, we get 
\begin{eqnarray}\label{Eq1}\nonumber
(Nv_\chi)_{(0,\infty)}&=&\frac{q^2}{2}\sum_{g\in
   T}\chi(g^{-1})\\
&-&\frac{q}{2}\Bigg(
\sum_{
\scriptsize
\begin{array}{c}
g \textrm{ s.t.   }\\
0^g=0
\end{array}
\normalsize
}\chi(g^{-1})+
\!\!\sum_{
\scriptsize
\begin{array}{c}
g \textrm{ s.t.   }\\
\infty^g=\infty
\end{array}
\normalsize
}\chi(g^{-1})
+
\sum_{
\scriptsize
\begin{array}{c}
g \textrm{ s.t.   }\\
0^g=\infty \\
\infty^g=0
\end{array}
\normalsize
}\!\!\!
\chi(g^{-1}) 
\Bigg).
\end{eqnarray}
On the right-hand side of Equation~\ref{Eq1}, the first, the second 
and the third summands are 
\begin{eqnarray*}
\frac{q^2}{2}(q-1)\langle \textrm{Res}_T(\chi),1\rangle
&=&\frac{q^2(q-1)}{2}\langle \chi,\pi\rangle=\frac{q^2(q-1)}{2},\\
\frac{q}{2}q(q-1)\langle \textrm{Res}_{(G_q)_0}(\chi),1\rangle
&=&\frac{q^2(q-1)}{2}\langle \chi,\lambda_1+\psi_1\rangle=0,\\
\frac{q}{2}q(q-1)\langle \textrm{Res}_{(G_q)_\infty}(\chi),1\rangle
&=&\frac{q^2(q-1)}{2}
\langle \chi,\lambda_1+\psi_1\rangle=0.\\
\end{eqnarray*}
Also, if $g$ is a permutation such that $0^g=\infty$ and $\infty^g=0$, then 
$g$ has order $2$. Therefore $g$ is conjugate to $u$ and so the fourth 
summand in Equation~\ref{Eq1} is $\frac{q(q-1)}{2}\chi(u)$. Summing
up, we obtain 
\[
(N_{v_\chi})_{(0,\infty)}=\left\{
\begin{array}{ccc}
\frac{(q^2-1)q}{2} &&\textrm{if }\chi=\eta_\beta,\\
\frac{(q-1)^2q}{2} &&\textrm{if }\chi=\nu_\gamma.
\end{array}
\right.
\]
Hence,  $(\ddag)$ yields  $s_{\nu_\gamma}=q(q-1)/2>0$ (for every $\gamma$) and 
$s_{\eta_\beta}=q(q+1)/2>0$ (for every $\beta$).

We point out that the matrix $N$ has only $4$ eigenvalues and actually
$N$ is the matrix of an association scheme of rank $4$, see~\cite{QX}
for more details.

Assume $q$ odd.  Let $\chi\in C$, with $\chi\neq
\lambda_1,\psi_1$. Now, as $\sum_{g\in G_q}\chi(g^{-1})=0$, subtracting
$(q-1)/2\sum_{g\in G_q}\chi(g^{-1})$ from $(\dag)$ and using
Proposition~\ref{crossratio}, we get
\begin{eqnarray}\label{Eq2}\nonumber
(Nv_\chi)_{(0,\infty)}&=&\frac{q^2-1}{2}\sum_{g\in T}\chi(g^{-1})-\frac{q-1}{2}
\Bigg(
\sum_{
\scriptsize
\begin{array}{c}
g \textrm{ s.t. }\\
0^g=0\\
\end{array}
\normalsize
}\chi(g^{-1})+
\sum_{
\scriptsize
\begin{array}{c}
g \textrm{ s.t. }\\
\infty^g=\infty\\
\end{array}
\normalsize
}\chi(g^{-1})\Bigg)\\
&&+\sum_{
\scriptsize
\begin{array}{c}
(a,b) \in \Omega\\
\crr(0,b,a,\infty) \textrm { not square}\\
\end{array}
\normalsize
}
\sum_{
\scriptsize
\begin{array}{c}
g \textrm{ s.t. }\\
0^g=a,\infty^g=b\\
\end{array}
\normalsize
}\chi(g^{-1}).
\end{eqnarray}
Arguing as in the case of $q$ even, we get that the first three 
summands in Equation~\ref{Eq2} are $(q^2-1)(q-1)/2$, $0$ and $0$.
Now, consider the subset
$\Delta=\{(a,b)\in\Omega\mid \crr(0,b,a,\infty) \textrm{ not a
  square}\}$ of $\Omega$ and the subset $S=\{g\in
G_q\mid (0,\infty)^g\in\Delta\}$ of $G_q$. Since
there are $(q-1)^2/2$ elements in $\Delta$ we have that $|S|=(q-1)^3/2$. 

If $\chi=\psi_{-1}$, then $|\chi(g^{-1})|\in\{0,1\}$, for every $g\in
S$. From Equation~\ref{Eq2} and the previous paragraph, we have
$(Nv_\chi)_{(0,\infty)}\geq (q^2-1)(q-1)/2-|S|>0$, so $s_{\psi_{-1}}>0$. 

If $(a,b)$ is in $\Delta$ and $g_{ab}$ is in $G_q$ such that 
$(0,\infty)^{g_{ab}}=(a,b)$,
then the set of elements of $G_q$ mapping $(0,\infty)$ to $(a,b)$ is the
coset $Tg_{ab}$. Since, $\crr(0,b,\infty,a)=\crr(0,b,a,\infty)^{-1}$
and $(a,b)\in\Delta$, we see that $\crr(0,b,\infty,a)$ is not a
square. So, by  
Proposition~\ref{crossratio}, we get that $Tg_{ab}$ contains exactly
$(q+1)/2$ elements conjugate to $v_r$, for some $r$. Therefore $Tg_{ab}$ 
contains at most $(q-3)/2$ elements conjugate to $d_x$, for some $x$.

Assume $\chi=\nu_\gamma$. Since $\chi(g)=0$ if $g$ is conjugate to $v_r$ 
(for some $r$), $|\chi(g)|\leq 2$ if $g$ is conjugate to $d_x$ 
(for some $x$) and $|\Delta|=(q-1)^2/2$, we obtain that the last summand
on the right-hand side of  Equation~\ref{Eq2} is greater than or equal to
$$\frac{(q-1)^2}{2}\cdot \frac{q-3}{2}\cdot
(-2)=-\frac{(q-1)^2(q-3)}{2}>-\frac{(q^2-1)(q-1)}{2}.$$ 
Hence $s_{\nu_\gamma}>0$.

Assume $\chi=\eta_\beta$. Now $|\chi(g)|\leq 2$ if $g$ is conjugate 
to $v_r$ (for some $r$), $\chi(g)=0$ if $g$ is conjugate to $d_x$ 
(for some $x$) and $|\Delta|=(q-1)^2/2$. Checking Table~\ref{Tb1}, we see  
that $-2=\chi(v_r)=-\beta(r)-\beta(r^{-1})$ if and only if
$r\mathbb{F}_q^*\in \Ker{\beta}$. Since $\beta$ has order greater than
$2$, we get $|\Ker\beta|<(q+1)/2$. Hence the last summand on 
the right-hand side of Equation~\ref{Eq2} is greater than 
$$\frac{(q-1)^2}{2}\cdot \frac{q+1}{2}\cdot
{(-2)}=-\frac{(q^2-1)(q-1)}{2}.$$  So, $s_{\eta_\beta}>0$. 
\end{proof}

Now, we construct a
submatrix $\overline{A}$ of $A$. In the submatrix 
$\overline{A}$, we keep all the
rows of $A$ and we delete the columns indexed by the ordered pairs
$(\infty,t)$, $(t,\infty)$, for every $t\in\mathbb{F}_q$. In
particular, we get  
\[
\overline{A}=\left(
\begin{array}{cc}
1&0\\
0&\overline{M}\\
B&\overline{C}
\end{array}
\right),
\] 
where the matrix $\overline{M}$ and $\overline{C}$ are obtained by
deleting the appropriate columns of $M$ and $C$. The matrix
$\overline{A}$ has $q^2+1$ columns and $\overline{M}$ has $q(q-1)$
columns. 

\begin{proposition}\label{rankMbar}We have
  $\mathrm{rank}(A)=\mathrm{rank}(\overline{A})$ and
  $\overline{M}$ has full column rank.  
\end{proposition}

\begin{proof}
We start by proving  that the columns indexed by $(\infty,t)$,
$(t,\infty)$ (for $t\in \mathbb{F}_q$) of $A$ are a linear combination
of the columns of  $\overline{A}$. We denote by $a_{xy}$ the column of
$A$ indexed by the 
ordered pair $(x,y)$, for $x,y\in
\mathbb{P}_q$. Since $G_q$ is $2$-transitive, it suffices to prove
that $a_{0\infty}$ is a linear combination of the columns of
$\overline{A}$. Set $v=\sum_{x\neq 0,\infty}\sum_{y\neq
  \infty}a_{xy}$ and
$w=(q-2)\sum_{x\neq\infty}a_{0x}+a_{\infty\infty}$. By
construction, $v$ and $w$ are a linear combination of the columns of
$\overline{A}$. Also, it is easy
to check that 
\[
v_g=\left\{
\begin{array}{lcl}
q-1&&\textrm{if }0^g=\infty \textrm{ or }\infty^g=\infty,\\
q-2&&\textrm{otherwise,}
\end{array}
\right.
\quad
w_g=\left\{
\begin{array}{lcl}
0&&\textrm{if }0^g=\infty,\\
q-1&&\textrm{if }\infty^g=\infty,\\
q-2&&\textrm{otherwise.}\\
\end{array}
\right.
\] 
Hence $(q-1)a_{0\infty}=v-w$ and we get that $a_{0\infty}$ is a linear
combination of the columns of $\overline{A}$. Thence
$\mathrm{rank}(A)=\mathrm{rank}(\overline{A})$. 

Since $\overline{M}$
has $q(q-1)$ columns, 
Proposition~\ref{rankM} shows that $\overline{M}$ has full column rank.
\end{proof}

\section{Proof of Theorem~\ref{main}. }\label{mainSec}

At this point, we have all the tools to conclude the proof of
Theorem~\ref{main}. 
By Lemma~\ref{boundPGL2q}, it remains to prove that if $S$ is an
independent set of maximal size of $\Gamma_{G_q}$, then $S$ is the coset
of the stabilizer of a point. Up to multiplication of $S$ by a suitable
element of $G_q$, we may assume that the identity element of $G_q$ is
in $S$. In particular, we have to prove that $S$ is the stabilizer of
a point. By Lemma~\ref{rankA} and Proposition~\ref{rankMbar}, we have
that the characteristic vector $v_S$ of $S$ is a linear combination of
the columns of $\overline{A}$. Hence
\[
\left(
\begin{array}{cc}
1&0\\
0&\overline{M}\\
B&\overline{C}
\end{array}
\right)
\left(
\begin{array}{c}
v\\
w
\end{array}
\right)=v_S,
\]
for some vectors $v,w$. As the identity element of $G_q$ is in $S$ and
by the ordering of the rows of $\overline{A}$, we get
\[
v_S=\left(
\begin{array}{c}
1\\
0\\
t
\end{array}
\right).
\] 
So, $1^T v=1$,  $\overline{M}w=0$ and $Bv+\overline{C}w=t$. By
Proposition~\ref{rankMbar}, the matrix 
$\overline{M}$ has full column rank. Thence  $w=0$ and $Bv=t$. 

Now, for every point $x$ of $\mathbb{P}_q$, there exists a permutation
$g_x$ of $G_q$ fixing $x$ and acting fixed-point-freely on
$\mathbb{P}_q\setminus\{x\}$ (indeed, $g_x$ can be chosen any
non-identity unipotent matrix of $G_q$ fixing $x$). Order the rows of
$B$ so that the first $q+1$ rows are labelled by the permutations
$\{g_x\}_x$. In particular, up to permuting the rows of $B$, we get 
\[
B=\left(
\begin{array}{c}
I_{q+1}\\
B'
\end{array}
\right)
\quad \textrm{and}\quad
Bv={v\choose B'v}.
\]
Since $Bv$ is equal to the $\{0,1\}$-vector $t$, we obtain that $v$ is a
$\{0,1\}$-vector. But $1^T v=1$ and so $v$ must be the characteristic
vector of a point $p$ of $\mathbb{P}_q$.  This shows that $v_S$ is the
stabilizer of the point $p$ and the proof is 
complete. 

\section{Comments}\label{comments}

Theorem~\ref{main} proves that in the derangement graph $\Gamma_{G}$,
where $G$ is the group $\PGL(2,q)$, the independent sets of maximal
size are the cosets of the stabilizer of a point. The same result holds if
$G$ is the symmetric group~\cite{CaKu, LaMa} or if $G$ is the
alternating group~\cite{KuWo}. It is interesting to ask for which
other permutation groups does this result hold, and is there a way to
characterize the permutation groups that have this property?

In Lemma~\ref{boundPSL2q}, we saw that the cosets of the stabilizer of
a point are independent sets of maximal size in the derangement graph
of $\PSL(2,q)$. We further conjecture that, similar to the case for
$\PGL(2,q)$, $\Sym(n)$ and $\Alt(n)$, all independent sets of
maximal size in $\Gamma_{\PSL(2,q)}$ are cosets of the stabilizer of a
point.

\begin{conjecture}
Every independent set $S$ of the derangement graph of
  $\PSL(2,q)$ acting on the projective line $\mathbb{P}_q$ has size at
  most $q(q-1)/2$. Equality is met if and only if $S$ is the coset
  of the stabilizer of a point.
\end{conjecture}

It seems likely that the methods used in this paper could be applied
to $\PSL(2,q)$, since the character table of this group is well
understood. 

In our proof of Theorem~\ref{main}, we used the fact that the group
$\PGL(2,q)$ is 2-transitive (for example, in the proof of
Proposition~\ref{crossratio} and Proposition~\ref{rankMbar}). It is a
reasonable question to ask if this (being $2$-transitive) could be a
characterization of the groups that have the property that the
independent sets of maximal size in the derangement graph are the
cosets of the stabilizer of a point. It is not hard to show that this is
not a characterization, and it is further interesting to see how this
property can fail to hold for some $2$-transitive groups.

If $G$ is a $2$-transitive group, then
the permutation character of $G$ is the sum of the
trivial character and an irreducible character that we will call
$\psi$. The eigenvalue of the derangement graph $\Gamma_G$ arising
from $\psi$ is $-\frac{d}{\psi(1)}$, where $d$
is the valency of $\Gamma_G$. If this eigenvalue is indeed
the least eigenvalue of $\Gamma_G$, then by
Lemma~\ref{bound}, we have that the size of an independent set is
no bigger than the size of the
stabilizer of a point in $G$. Thus, if it were true that the
eigenvalue arising from the character $\psi$ is the least eigenvalue,
then we would have the bound like in Theorem~\ref{main}. But the
characterization of the sets that meet this bound is another question
entirely. In fact, there are examples of $2$-transitive groups in
which there are independent sets of maximal size that are not cosets
of the stabilizer of a point.

For example, let $G_{n,q}$ be the $2$-transitive group
$\PGL(n+1,q)$ in its action on the projective space
$\mathbb{P}_q^n$, with $n\geq 1$. 
Since $G_{n,q}$ contains a Singer cycle of length
$(q^{n+1}-1)/(q-1)$, the graph $\Gamma_{G_{n,q}}$ has a clique  of
size $|\mathbb{P}_q^n|=(q^{n+1}-1)/(q-1)$.  Thus, any independent set
of $\Gamma_{G_{n,q}}$ has size at most $|G_{n,q}|/|\mathbb{P}_q^n|$.
Naturally, the stabilizer of a point is an independent set for
$\Gamma_{G_{n,q}}$ of size $|G_{n,q}|/|\mathbb{P}_q^n|$
and so it is an independent set of maximal size.  But for $G_{n,q}$
with $n\geq 2$, 
the cosets of the stabilizer of a point are not the only independent
sets of maximal size. Indeed, it is not hard to see that the
stabilizer of a hyperplane of $\mathbb{P}_q^n$ in $G_{n,q}$ is also an
independent set of maximal size for $\Gamma_{G_{n,q}}$. Moreover, if
$n\geq 2$, then the stabilizer of a point and of a
hyperplane are not conjugate subgroups of $G_{n,q}$. Therefore, for
$n\geq 2$, the graph $\Gamma_{G_{n,q}}$ contains at least
$2(|G_{n,q}|/|\mathbb{P}_q^n|)^2$ independent sets of maximal
size. We make the following conjecture.

\begin{conjecture}Any independent set of maximal size in the
  derangement graph of $\PGL(n+1,q)$ acting on the projective
  space $\mathbb{P}_q^n$ is either the coset of the stabilizer of a
  point or the coset of the stabilizer of a hyperplane.
\end{conjecture}

It is not clear that the method in this paper could be used to prove
this conjecture.  In particular, even if the character table of
$\PGL(n+1,q)$ is known~\cite{Gr}, it would still be challenging to
obtain the minimum eigenvalue of $\Gamma_{G_{n,q}}$ as in
Lemma~\ref{boundPGL2q}.

Further, there exist $2$-transitive groups of degree $n$ where the
number of independent sets of maximal size is $n^{n-1}$.
Clearly, this means that there are many 
independent sets of maximal size which are not cosets of the stabilizer
of a point. For example, let $n$ be a power of a prime and
$\mathbb{F}_n$ the field with $n$ 
elements. The affine 
general linear group $G$ on $\mathbb{F}_n$ (i.e. the group generated by
the permutations of $\mathbb{F}_n$ of the form $f_{a,b}:\xi\mapsto a\xi+b$,
for $a,b\in\mathbb{F}_n$ and $a\neq 0$) is a $2$-transitive
group. Since $G$ is a Frobenius group with kernel of size $n$ and
complement of size $n-1$, we
obtain that $\Gamma_G$ is the disjoint union of $n-1$ complete graphs
$K_n$. In particular, we get that $\Gamma_G$ has
$n^{n-1}$ independent sets of maximal size.

From the above comments it is clear that a result similar to
Theorem~\ref{main} does not hold
for all $2$-transitive groups, but perhaps we do better to consider
3-transitive groups. 
In particular, some information on the $3$-transitive groups is listed
in Table~\ref{Tb3} and Table~\ref{Tb4},
see~\cite[page~$195$,~$197$]{Ca}. Table~\ref{Tb3} lists  all possible
socles of an almost simple $3$-transitive 
group and Table~\ref{Tb4} gives all possible point stabilizers of
an affine $3$-transitive group.

\begin{table}[!h]
\begin{tabular}{l|ccccccccc}
Group & $M_{11}$&$M_{11}$&$M_{12}$&$M_{22}$&$M_{23}$&$M_{24}$&
$\mathrm{PSL}(2,q)$&$\textrm{Alt}(n)$   \\\hline
Degree&   $11$ & $12$   & $12$   & $22$   & $23$   & $24$&$q+1$&$n$\\  
\end{tabular}\caption{Socle of an almost simple $3$-transitive group.}
\label{Tb3}
\begin{tabular}{l|cccc}
Group & $\mathrm{SL}(n,2)\leq H\leq \Gamma\mathrm{L}(n,2)$&
$\textrm{GL}(1,3)$&$\Gamma\textrm{L}(1,4)$&$\textrm{Alt}(7)$\\\hline
Degree&   $2^n$                                      &$3$&$4$&$16$ \\  
\end{tabular}\caption{Point stabilizer of an affine $3$-transitive
  group.}
\label{Tb4}
\end{table}
Using \texttt{GAP}, it is straight-forward to build the derangement
graph for each of the $3$-transitive groups of degree $11,12,22,23,24$
and $16$ and to then find all the independent sets of maximal size.
Indeed, for every one of these groups, every independent set of maximal
size of $\Gamma_G$ is the coset of the stabilizer of a point. Thus we
conclude with the following conjecture.
\begin{conjecture}Let $G$ be a $3$-transitive group of degree $n$. Every
  independent set $S$ of the derangement graph of 
  $G$ acting on $\Omega$ has size at most $|G|/n$. Equality is met if
  and only if $S$ is the coset of the stabilizer of a point.
\end{conjecture}


\newpage
\begin{thebibliography}{10}
\bibitem{Ba}L. Babai, Spectra of Cayley Graphs, J. of Combin.
  Theory B.  2, (1979), 180--189. 

\bibitem{Ca}P. J. Cameron, Permutation Groups, London Mathematical
  Society Student Texts 45, (1999). 

\bibitem{CaKu}P. J. Cameron, C. Y. Ku, Intersecting families of
  permutations, European J. Combin. 24 (7) (2003),  881--890. 

\bibitem{DeFr}M. Deza, P. Frankl, Erd\H{o}s–-Ko–-Rado theorem--22 years
  later, SIAM J. Algebraic Discrete Methods 4 (4) (1983),
   419--431. 

\bibitem{ErKoRa}P. Erd\H{o}s, C. Ko, R. Rado, Intersection theorems for
  systems of finite sets, Quart. J. Math. Oxford Ser. 12 (2) (1961),
   313--320. 

\bibitem{MR934346} P. Frankl, Intersection theorems for finite sets
  and geometric applications, Proceedings of the {I}nternational
  {C}ongress of {M}athematicians 1(2) (1986),  1419--1430.

\bibitem{GAP} The GAP Group, GAP -- Groups, Algorithms, and
  Programming, Version 4.4.12; 2008. (http://www.gap-system.org)

\bibitem{GoMe}C. Godsil, K. Meagher, A new proof of the
  Erd\H{o}s-Ko-Rado theorem for intersecting families of permutations,
  European J. of Combin. 30, (2009), 404--414.

\bibitem{Gr}J. A. Green, The characters of the finite general linear
  groups,  Trans. Amer. Math. Soc. 80, (1955), 402--447.

\bibitem{Is}M. I. Isaacs, Character Theory of Finite groups, AMS
  Chelsea Publishing, (2006).

\bibitem{JamesLiebeck}G. James, M. Liebeck, Representations and
  Characters of Groups, Cambridge University Press, (2001).

\bibitem{katona:73} G.~Katona, Two applications (for search theory and
  truth functions) of {S}perner type theorems, Period. Math. Hungar.
  3, (1973), 19--26.

\bibitem{kleitman:73} D.~J. Kleitman, J.~Spencer, Families of
  {$k$}-independent sets, Discrete Math. 6, (1973), 255--262.

\bibitem{KuWo} C. Y. Ku, T. W. Wong, Intersecting families in the
  alternating group and 
direct product of symmetric groups, Electronic J. of Combin. 14,
(2007), 15pp. 

\bibitem{LaMa}B. Larose, C. Malvenuto, Stable sets of maximal size
  in Kneser-type graphs, European J. Combin. 25(5) (2004),
   657--673. 

\bibitem{Li}T. M. Liggett, Extensions of the Erd\H{o}s-Ko-Rado theorem
  and a statistical application, J. Comb. Theory A 23, (1977),
  15--21.

\bibitem{QX}Q. Xiang, Association schemes from ovoids in
  $\mathrm{PG}(3,q)$, Journal of Geometry 76, (2003), 294--301.
\end{thebibliography}
\end{document}